\documentclass{birkjour}
 \newtheorem{theor}{Theorem}[section]
 
 \newtheorem{lem}[theor]{Lemma}
 \newtheorem{prop}[theor]{Proposition}
 \theoremstyle{definition}
 \newtheorem{defi}[theor]{Definition}
 \theoremstyle{remark}
 \newtheorem{rem}{Remark}
 \newtheorem{ex}{Example}
 \numberwithin{equation}{section}

\usepackage{amsmath}
\usepackage{amssymb}
\usepackage{mathtools}
\usepackage{amsfonts}

\def\r{r}
\def\p{p}
\def\ellm{{\ell_{-}}}
\def\ellp{{\ell_{+}}}
\def\ellpm{{\ell_\pm}}
\def\dx{\partial_x}
\def\dxx{\partial_{xx}}
\def\dt{\partial_t}
\def\he{\mathrm{Hess}}
\def\RR{\mathbb R}

\def\eps{\varepsilon}
\def\gg{\mathcal{G}}
\def\ll{\mathcal{L}}
\def\rr{\mathcal{R}}

\def\d{\mathrm{d}}

\def\de{\partial}
\def\Id{\operatorname{Id}}
\def\half{\frac{1}{2}}

\newcommand{\nm}[1]{\left\| #1 \right\|}

\newcommand{\quot}[1]{``#1''}

\begin{document}

\title{Global existence for a class of viscous  systems of conservation laws}

\author{Luca Alasio}
\address{Gran Sasso Science Institute, L'Aquila (Italy)}
\email{luca.alasio@gssi.it}

\author{Stefano Marchesani}
\address{Gran Sasso Science Institute, L'Aquila (Italy}
\email{stefano.marchesani@gssi.it}



\date{August 2019}

\begin{abstract}
We prove existence and boundedness of classical solutions for a family of viscous conservation laws in one space dimension for arbitrarily large time. The result relies on H. Amann's criterion for global existence of solutions and on suitable uniform-in-time estimates for the solution.
We also apply J\"ungel's boundedness-by-entropy principle in order to obtain global existence for systems with possibly degenerate diffusion terms.
This work is motivated by the study of a physical model for the space-time evolution of the strain and velocity of an anharmonic spring of finite length.
\end{abstract}

\maketitle

\section{Introduction and motivation}\label{sec:intro}

In this paper we study global existence and uniqueness of solutions for a family of parabolic systems of PDEs in one space dimension.
Although the literature concerning parabolic problems is very rich, results for nonlinear systems subject to non-homogeneous boundary conditions of different type (Dirichlet, Neumann or mixed) are not always available, even if the spatial domain is just an interval of the real line.
Moreover, many of the classical results for parabolic systems are formulated on a fixed time interval $[0,T]$ (as in \cite{ladyzhenskaia1988linear}), whereas we are interested in obtaining  estimates for the solutions that are valid for arbitrarily large time.
The study of global properties becomes trickier when the system includes cross-diffusion and possibly degenerate terms, as they affect the regularizing effects of diffusion terms.
We choose to restrict our attention to the case $d=1$ in order to make the exposition clearer and to give neat statements. Indeed, working in one dimension offers several advantages in terms of regularity results and Sobolev embeddings.
In some cases it is possible to extend many of the results we present to $d>1$ (see e.g. \cite{alasio2018stability, alasio2019trend}), however, in general, one can not guarantee existence of global, classical solutions of strongly coupled systems (see, for example, \cite{mooney2017, stara1995some}).
The system we consider is complemented by mixed, time-dependent boundary conditions which are imposed at the extreme points of the domain.
Hyperbolic and parabolic systems of conservation laws with mixed and time-dependent boundary conditions arise naturally when studying the thermodynamics of some microscopic systems (see for example \cite{olla2014microscopic, olla2015microscopic}). 
Our choice of boundary conditions was motivated by the case of an anharmonic chain of particles connected by nonlinear springs, where one end of the chain is fixed (homogeneous Dirichlet) and at the other end is applied a constant force (non-homogeneous Dirichlet).
By means of such boundary conditions and a suitable choice of the external force, it is possible define thermodynamic transformations and deduce the first and second law of Thermodynamics (macroscopic laws) as consequences of the microscopic dynamics. 
We refer to \cite{even2010hydrodynamic, MO1, MO2, olla2014microscopic} for further details concerning these physical models.
A prototypical model for our study is given by the following viscous p-system obtained as hydrodynamic limit (under hyperbolic space-time scaling) for the isothermal dynamics of an anharmonic chain subject to an external varying tension:
\begin{ex}\label{example}
As described in \cite{SM2}, a suitable choice of the microscopic model leads to the following viscous p-system
\begin{equation} \label{eq:psystem}
    \begin{cases}
        \de_t\r  =\de_x\p+ \delta_1 \de_{xx}\tau(\r)
        \\
        \de_t\p =\de_x\tau(\r) + \delta_2\de_{xx} \p
    \end{cases}, \qquad (t,x) \in (0,\infty) \times (0,1)
\end{equation}
with boundary conditions
\begin{equation}
    \p(t,0) = 0, \quad \tau(\r(t,1)) = a(t), \quad \de_x \p(t,1)=\de_x \r(t,0)=0.
\end{equation}
Here $\r$ and $\p$ represent infinitesimal strain and velocity of each point  of a anharmonic spring of finite length, $\tau$ is the internal tension, and $a$ represents the boundary tension.
\end{ex}
A second example of system we can study is the following:
\begin{ex}
Consider the $4\times 4$ system of PDEs given by
\begin{align} \label{eq:sysmax}
    \begin{cases}
    \de_t E_y = -\de_x B_z + \delta_1 \de_{xx}E_y
    \\
    \de_t E_z= \de_x B_y+ \delta_2 \de_{xx}E_z
    \\
    \de_t B_y = \de_x E_z + \delta_3 \de_{xx}B_y
    \\
    \de_t B_z = - \de_x E_y+ \delta_4 \de_{xx}B_z
    \end{cases}
    \; ,
    \end{align}
    with boundary conditions
    \begin{align}
        \mathbf E(t,0)=0,\quad \mathbf B(t,1)= \mathbf b(t), \quad \de_x \mathbf E(t,1)=0,  \quad \de_x \mathbf B(t,0)=0.
    \end{align}
    Here $\mathbf E =(E_y, E_z)$ and $\mathbf B =(B_y, B_z)$ can be interpreted as the $y$ and $z$ components of an electric and magnetic field in the vacuum.
    When $\delta_i =0$, system \eqref{eq:sysmax} reduces to the  $y$ and $z$ components of the Maxwell equations
    \begin{align}
        \de_t \mathbf E = \operatorname{curl} \mathbf B, \qquad \de_t \mathbf B =- \operatorname{curl} \mathbf E,
    \end{align}
    in the case where $\mathbf E$ and $\mathbf B$ propagate along the $x$-axis (that is, when they do not depend on $y$ and $z$). The Dirichlet boundary conditions fix the values of the electric field at $x=0$ and of the magnetic field at $x=1$. The Neumann conditions are no-flux conditions.
\end{ex}

\section{Set up and main results}

Consider the interval $I=(\ellm, \ellp) \subset \RR$, where $\ellm < \ellp$, and let  $\xi^k$ denote the $k$-th component of a generic vector $\xi\in\RR^N$. 
We denote by $Q_T$ the parabolic cylinder $[0,T)\times I$.
For $k=1\dots N$, we consider the system of PDEs
\begin{equation}\label{eq:sysgen0}
\de_t u^k = \de_x \sum_{l=1}^N\left(
M^{kl} F^l(u) + A^{kl}\de_x F^l(u)
\right),
\end{equation}
where the unknown $u$ is a vector function of the independent variables time $t$ and space $x$, namely  $u:[0,T]\times \bar I\to\RR^N$, $F:\RR^N\to\RR^N$ is a vector field of class $C^1$.
The matrices $A=\{A^{kl}\}$ and $M=\{M^{kl}\}$, for $k,l=1,\dots N$, will be defined precisely later on. 
It is sometimes convenient to write equation \eqref{eq:sysgen0} in matrix form, in particular
\begin{equation} \label{eq:sysmatrix}
\de_t u = \de_x \left(
M F(u) + A \de_x F(u)
\right).
\end{equation} 
We shall also denote by $\cdot$ the scalar product of $\RR^N$.

Equation \eqref{eq:sysgen0} is complemented by the following boundary conditions:
\begin{align}\label{eq:bc}
F^1(u(t,\ellp))  = a(t),
&\qquad (\de_xF^1)(u(t,\ellm)) = 0, &
\nonumber \\
(\de_xF^2)(u(t,\ellp)) = 0,
&\qquad F^2(u(t,\ellm)) = 0, &
\\
F^k(u(t,\ellp)) = 0,
&\qquad F^k(u(t,\ellm)) = 0, & \text{ for } k > 2
\nonumber
\end{align}
and by the initial condition
\begin{equation}\label{eq:ic}
    u^k(0,x) = u^k_0(x).
\end{equation}

\begin{rem}
It is possible to replace \eqref{eq:bc} with other boundary conditions such as homogeneous Dirichlet or assigned periodic conditions and our results still hold with different constants.
\end{rem}

\noindent We consider the following set of assumptions:
\begin{enumerate}
\item[\bf{[H1]}]
We assume that $A$ and $M$ are $N\times N$ symmetric matrices. We suppose that $A=A(x)$ is (at least) of class $C^1$ with respect to $x\in \Bar{I}$ and that there exists a constant $\mu>0$ such that, for any $\xi\in\RR^N$, $x\in\bar{I}$, it holds
$$\mu |\xi|^2\leq A(x)\xi\cdot\xi \leq \frac{1}{\mu}|\xi|^2.$$
We also impose the following compatibility conditions (for $k=1\dots N$):
\begin{align}
    A^{1k}(l_-)= 0,\quad &k \neq 1
    \\
    A^{2k}(l_+)= 0,\quad & k \neq 2.
\end{align}
We further assume that $M=M(t)$ is (at least) of class $C^\alpha$, for some $\alpha\in(0,1/2)$, bounded with respect to $t\in [0,\infty)$ and we require that 
$$
M^{kk}=0, \qquad \forall k=1,\cdots,N.
$$
\item[\bf{[H2]}]
The function $a:[0,\infty)\to\RR$ satisfies the following conditions:
$$
\nm{a}_{L^\infty(0,\infty)}\leq a_0,
\qquad
\left(
\int_0^\infty |a'(t)|^p \d t
\right)^{\frac{1}{p}} \leq a_1,
$$
for any $p\in[1,2]$ and suitable constants $a_i>0$, $i=0,1$.
\item[\bf{[H3]}]
The initial condition $u_0:\bar I\to\RR^N$ is (at least) of class $C^{2}$ and it is compatible with the boundary conditions \eqref{eq:bc}.
\item[\bf{[H4]}]
There exists a convex function $H:\RR^N\to \RR_+$ of class $C^2$ and a constant $\lambda>0$ such that 
$$
DH(u) = F(u)
\quad \text{ and } \quad 
\he(H)(u)\xi\cdot\xi \geq \lambda |\xi|^2, 
\quad \forall u,\xi \in\RR^N.
$$
Furthermore, $F:\RR^N \to \RR^N$ is monotone of class $C^1$  and there exists a constant $\Lambda\geq\lambda$ such that, for any $ \xi_1,\xi_2\in\RR^N$,
\begin{equation}\label{eq:monotone}
\lambda|\xi_1-\xi_2|^2
\leq 
(F(\xi_1) - F(\xi_2))\cdot (\xi_1 - \xi_2)
\leq
\Lambda|\xi_1-\xi_2|^2.
\end{equation}
We also assume that $F(0)=0$.
\end{enumerate}

\begin{rem}[Initial datum in $L^2$]
If the initial datum is not smooth but, for example, it only belongs to $L^2(I)$, the initial time has to be excluded but our results still hold in a subset of the form $(t_0, T)$, for $t_0$ arbitrarily close to $0$.
\end{rem}
\begin{rem}
Model \eqref{eq:psystem}  is obtained as a special case of system \eqref{eq:sysgen0} setting
$$
u = \begin{pmatrix} r \\ p\end{pmatrix},
\quad
F(u) = \begin{pmatrix} \tau(r) \\ p\end{pmatrix},
\quad
M = \begin{pmatrix} 0 & 1 \\ 1 & 0\end{pmatrix}, 
\quad
A = \begin{pmatrix} \delta_1 & 0 \\ 0 & \delta_2\end{pmatrix}.
$$
Similarly, model \eqref{eq:sysmax} is obtained from system \eqref{eq:sysgen0} for $F(u)=u$ setting
$$
u = \begin{pmatrix} E_y \\ E_z \\ B_y \\ B_z\end{pmatrix},
\quad
M = \begin{pmatrix} 0 & 0 & 0 & -1 \\ 0 & 0 & 1 & 0 \\ 0 & 1 & 0 & 0 \\ -1 & 0 & 0 & 0\end{pmatrix}, 
\quad
A = \begin{pmatrix} \delta_1 & 0 & 0 & 0 \\ 0 & \delta_2 & 0 & 0 \\ 0 & 0 & \delta_3 & 0\\ 0 & 0 & 0 & \delta_4\end{pmatrix}.
$$
\end{rem}
Our main result is the following:
\begin{theor}\label{thm:main}
Under the hypotheses {\bf{[H1]}-\bf{[H4]}}, problem \eqref{eq:sysgen0}-\eqref{eq:ic} admits a bounded, global, classical solution in $C^1([0,\infty);C^0(\bar{I})) \cap C^0([0,\infty);C^2(\bar{I}))$.
Furthermore, there exists a constant $C>0$ independent of time such that
\begin{align}
\int_I &
 \left(
 |u(T,x)|^2 + 
\mu|\dx F(u(T,x))|^2
 \right)\d x 
 \nonumber \\ &+
 \int_{Q_T}
\left(
\mu |\dx F(u)|^2
+
 |\dt u|^2
+
|\dx (A \dx F(u))|^2
\right)
 \d x \d t
 \leq C.
\end{align}
\end{theor}

In Section \ref{sec:degenerate} we will briefly discuss a way to extend global existence results to a family of possibly degenerate systems (under suitable assumptions); this technique was established in \cite{jungel2015boundedness} (see also \cite{jungel2016entropy}).

Theorem \ref{thm:main}, will be proved in Section \ref{sec:estimates} and we will use the following two fundamental \quot{building blocks}:

\begin{theor}[Existence of classical solutions for short times, \cite{acquistapace1988quasilinear}]\label{thm:AT}
Let hypotheses {\bf{[H1]}-\bf{[H4]}} hold, then there exists a time $t_1\in(0,T)$ such that, given a sufficiently smooth initial datum $u_0$ that is compatible with the boundary conditions, problem \eqref{eq:sysgen0}-\eqref{eq:ic} has a unique solution $u$ in the interval $[0,t_1]$ which satisfies
$$
u\in C^{1+\alpha}([0,t_1];L^2(I))\cap C^{\alpha}([0,t_1];H^2(I)),
$$
where $\alpha\in (0,\frac{1}{2})$ depends on the regularity of $u_0$ and $M$.
Furthermore, we have
$$
\de_t u, \, \dxx u \in C^{\alpha_1}([0,t_1];C^0(\bar{I})),
$$
for each $\alpha_1\in(0,\alpha)$.
\end{theor}
The following theorem, combined with uniform-in-time estimates, allows to show global existence of classical solutions for a relatively wide class of cross-diffusion systems (see, for example, \cite{alasio2018stability, alasio2019trend} and references therein).
\begin{theor}[Criterion for global existence,  \cite{amann1989dynamic}]\label{thm:amann}
Let hypotheses {\bf{[H1]}-\bf{[H4]}} hold and consider a solution $u$ of problem \eqref{eq:sysgen0}-\eqref{eq:ic}. Let $J(u_0)$ denote the maximal time interval of definition of $u$.
If there exists an exponent $\eps\in (0,1)$ (not depending on time) such that
$$
u\in C^\eps(J(u_0)\cap[0,T];C^0(\bar{I})),
$$
Then $u$ is a global solution, i.e. $J(u_0)=[0,\infty)$.
\end{theor}
\begin{rem}[Notation]
Notice that, comparing Theorem \ref{thm:amann} with the original statement in \cite{amann1989dynamic}, we have that $\gg = \RR^N$.
Additionally, in our case the function $f$ introduced in \cite{amann1989dynamic} is \quot{affine in the gradient} in the sense specified therein. Finally, we do not use the notation $BUC^\eps$ to denote the space of \quot{bounded, uniformly $\eps$-H\"older continuous functions}.
\end{rem}

\section{Estimates for the general system}\label{sec:estimates}

In the present section we are going to derive the crucial estimates for solutions of system \eqref{eq:sysgen0}.

\begin{prop}\label{prop:est1}
Let hypotheses {{\bf[H1]}-{\bf[H4]}} hold.
There exists a constant $C_0>0$ independent of $A$ such that for any $T >0$, any solution $u$ of \eqref{eq:sysgen0}-\eqref{eq:ic} satisfies
\begin{equation}
 \int_I |u(T,x)|^2 \d x 
 +\mu \int_0^T \int_I | \de_x F(u)|^2 \d x \d t \le C_0.
\end{equation}
\end{prop}
\begin{proof}
Thanks to Theorem \ref{thm:AT}, we know that classical solutions exist on a (possibly short) time interval $[0,t_1]$, and therefore the maximal time interval of existence of $u$, denoted by $J(u_0)$, is well defined.

Let $t\in [0,T] \subseteq J(u_0)$; we test \eqref{eq:sysmatrix} against $F(u)$ and integrate over $I$:
\begin{equation} \label{eq:u1}
\int_I F(u) \cdot \de_t u \d x 
= \int_I \left\{ F(u) \cdot M \de_x F(u) +  F(u) \cdot \de_x (A \de_x F(u)) \right\} \d x
\end{equation}
Recall that, by assumption {\bf[H4]}, $F$ has a primitive, namely $F(u)= DH(u)$, where $H$ is a convex and non-negative scalar function. Therefore, we can write
$$
\int_I F(u) \cdot \de_t u \d x = \int_I \de_t H(u) \d x.
$$
Moreover,  since $M$ is symmetric and independent of $x$, we have
\begin{align*}
     F(u) \cdot M \de_x F(u) = \frac{1}{2} \de_x (F(u) \cdot M F(u)).
\end{align*}
Thus,  after an integration by parts, \eqref{eq:u1} becomes
\begin{align}\label{eq:est1}
 & \int_I(\de_t H(u) +F(u)\cdot A \de_x F(u) )\d x 
 \nonumber \\
 &= \left( \frac{1}{2}F(u) \cdot M F(u)+F(u) \cdot A \de_x F(u) \right) \biggr |_{\de I}.
\end{align}
We evaluate the boundary term in \eqref{eq:est1}. 
Since we have homogeneous Dirichlet boundary conditions
$
F^k(u(t,\ellm)) = F^k(u(t,\ellp))=0,
$
for $k > 2$, we obtain
\begin{align*}
&  \left( \frac{1}{2}F(u) \cdot M F(u)+F(u) \cdot A \de_x F(u) \right) \biggr |_{\de I}
\\
&= \left(\frac{1}{2}\sum_{l,k=1}^N M^{kl} F^k(u) F^l(u) +  \sum_{l,k=1}^N A^{lk} F^k(u) \de_x F^l(u) \right)\biggr |_{\de I}
\\
&=  \left( M^{12}F^1(u)F^2(u) + \sum_{l=1}^N (A^{l1}F^1(u)\de_xF^l(u)+A^{l2}F^2(u)\de_xF^l(u)) \right) \biggr |_{\de I}.
\end{align*}
Moreover, since $\de_x F^1(u(t,\ellm))=F^2(u(t,\ellm))=\de_x F^2(u(t,\ellp)) = 0$ and $F^1(u(t,\ellp)) = a(t)$, we have
\begin{align*}
& \left( M^{12}F^1(u)F^2(u) + \sum_{l=1}^N (A^{l1}F^1(u)\de_xF^l(u)+A^{l2}F^2(u)\de_xF^l(u)) \right) \biggr |_{\de I}
\\
&=  a(t) \left( M^{12} F^2(u(t,\ellp))+ \sum_{l \neq 2}A^{l1}(\ellp) \de_x F^l(u(t,\ellp)) \right)
\\
&\qquad + \sum_{l\neq 2} A^{l2}(\ellp)F^2(u(t,\ellp))\de_x F^l(u(t,\ellp))
\\
& \qquad-\sum_{l \neq 1} A^{l1}(\ellm)F^1(u(t,\ellm))\de_x F^l(u(t,\ellm)).
\end{align*}
Finally, since $A^{l2}(\ellp)=0$ for $l \neq 2$ and $A^{l1}(\ellm)=0$ for $l \neq 1$, we obtain
\begin{align}\label{eq:survivingbc}
    &\left( \frac{1}{2}F(u) \cdot M F(u)+F(u) \cdot A \de_x F(u) \right) \biggr |_{\de I} \nonumber
    \\
    & = a(t) \left( M^{12} F^2(u(t,\ellp))+ \sum_{l \neq 2}A^{l1}(\ellp) \de_x F^l(u(t,\ellp)) \right).
\end{align}
We deduce the value of the last bracket in \eqref{eq:survivingbc} by integrating equation \eqref{eq:sysgen0} for $k=1$ with respect to $x\in I$:
\begin{equation}
M^{12} F^2(u(t,\ellp))+ \sum_{l \neq 2}A^{l1}(\ellp) \de_x F^l(u(t,\ellp)) = \de_t \int_I u^1 \d x.
\end{equation}
Given the boundary terms above, integrating \eqref{eq:est1} in time, we obtain
\begin{align} \label{eq:split}
\begin{split}
 &    \int_I  H(u(T,x))  \d x+ \int_0^T \int_I \de_x F(u) \cdot A \de_x F(u) \d x \d t 
 \\
& =\int_I H(u_0) \d x+ \int_0^T a(t) \de_t \int_I u^1 \d x \d t
\\
 &     =\int_I H(u_0) \d x+\left( a(t) \int_I u^1\d x \right)\biggr |_{t=0}^{t=T} - \int_0^T a'(t) \int_I u^1\d x \d t.
      \end{split}
\end{align}
Let us estimate the last two terms in \eqref{eq:split}, in particular we have
\begin{align}\label{eq:estc}
a(T)\int_I u^1(T,x)\d x  
&\le \frac{\ellp-\ellm}{\lambda}
a(T)^2 +\frac{\lambda}{4}\int_I (u^1(T,x))^2\d x,
\end{align}
and, using Young's inequality,
\begin{equation}\label{eq:est}
-\int_0^T a'(t) \int_I u^1 \d x \d t \le  
\frac{\ellp-\ellm}{4}\int_0^T |a'(t)| \d t 
+ \int_0^T |a'(t)| \int_I(u^1)^2 \d x \d t.
\end{equation}
We recall that, by assumption {\bf [H1]}, it holds
$
\de_x F(u) \cdot A \de_x F(u) \ge \mu |\de_x F(u)|^2
$
and, additionally, by assumption {\bf [H4]}, we have
$
H(u) \ge \frac{\lambda}{2} |u|^2.
$
Therefore, combining inequalities \eqref{eq:split}, \eqref{eq:estc} and \eqref{eq:est},
we obtain the following estimate:
\begin{align}\label{eq:estd}
\frac{\lambda}{4} \int_I 
&|u(T,x)|^2 \d x +\int_0^T \int_I \mu|\de_x F(u)|^2 \d x \d t 
\nonumber \\
&\le C(T) 
+ \int_0^T |a'(t)| \int_I (u^1)^2 \d x \d t,
\end{align}
where, using assumption 
{\bf [H2]}, we have
\begin{align*}
C(T) 
&= \int_I H(u_0) \d x
+
(\ellp-\ellm)\left(
\frac{1}{\lambda}
a(T)^2
+
\frac{1}{4}\int_0^T |a'(t))| \d t
\right)
\\
&\leq
\int_I H(u_0) \d x
+
(\ellp-\ellm)\left(
\frac{1}{\lambda}
a_0^2
+
\frac{1}{4}a_1
\right).
\end{align*}
Finally, thanks to \eqref{eq:estd}, we apply Gr\"onwall's inequality
and obtain 
\begin{align}
\frac{\lambda}{4} \int_I 
&
|u(T,x)|^2 \d x 
+\mu \int_0^T \int_I |\de_x F(u)|^2 
\d x \d t 
\nonumber \\
&\leq 
C(T) 
\exp\left(\int_0^T |a'(t)| \d t\right)
\nonumber \\
&\leq 
\left(
\int_I H(u_0) \d x
+
(\ellp-\ellm)\left(
\frac{1}{\lambda}
a_0^2
+
\frac{1}{4}a_1
\right)
\right)
\exp(a_1).
\end{align}
Thus each term is bounded by a constant independent of $T$.
\end{proof}
In the next Proposition we will obtain stronger estimates involving first derivatives in time and second derivatives in space.
\begin{prop}\label{prop:est2}
Let hypotheses {\bf{[H1]}-\bf{[H4]}} hold and consider a solution $u$ of \eqref{eq:sysgen0}-\eqref{eq:ic}.
Assume that $\he(H) \geq \lambda \Id$, where
$\lambda > \frac{1}{4\sigma} > \frac{1}{2}$,
for some $\sigma\in (0,\half)$.
Then there exists a constant $C_1>0$ independent of $T$ (but depending on $A,M,u_0, C_0, a, \sigma, \mu, \lambda)$ such that
\begin{equation}
\int_{Q_T}
\left(
 |\dt u|^2
+
|\dx (A \dx F(u))|^2
\right)
\d x\d t
+
\int_I
|\dx F(u)|^2\big|_{t=0}^{t=T}
\d x
\leq
C_1.
\end{equation}
Additionally, $u\in L^2(0,T;H^2(I)) \cap H^1(0,T;L^2(I))$ uniformly for all $T>0$.
\end{prop}
\begin{rem}
Notice that the condition $\lambda>\frac{1}{2}$ can be removed by re-scaling the time variable in equation \eqref{eq:sysgen0}.
\end{rem}
\begin{proof}
Given $T\in J(u_0)$, we test the general system
\begin{equation}\label{eq:sysagain}
\de_t u = \de_x \left(
M F(u) + A \de_x F(u)
\right)
\end{equation}
(with $A=A(x)$ and $M=M(t)$) against $\dt F(u) - \Xi$, where $\Xi = \dx (A \dx F(u))$; namely we obtain
\begin{align}\label{eq:testeq}
& \int_{Q_T}
\left[\dt u - \Xi \right]
\cdot
\left[\dt F(u) - \Xi \right]
\d x\d t
\nonumber \\
&=
\int_{Q_T}
\dx (MF(u))
\cdot
\left[\dt F(u) - \Xi \right]
\d x\d t.
\end{align}
Let us denote the left-hand side and right-hand side of \eqref{eq:testeq} by $\ll$ and $\rr$ respectively. We are going to estimate the following term from below:
$$
\ll =
\int_{Q_T}
\left[
\he(H)  \dt u \cdot  \dt u
+
|\Xi|^2
-
(\dt u + \dt F(u))\cdot \Xi \right]
\d x\d t,
$$
in particular, we estimate the two \quot{mixed terms} separately. For the first one we have 
$$
-\int_{Q_T} \dt u \cdot \Xi
\d x\d t
\geq 
- \int_{Q_T}
\left[
\sigma \lambda
|\dt u |^2
+
\frac{1}{4\sigma \lambda}
|\Xi|^2
\right]
\d x \d t,
$$
whereas for the second one we have 
\begin{align*}
-\int_{Q_T} \dt F(u) \cdot \Xi
\d x\d t
&=
-\int_{Q_T} \dt F(u) \cdot \dx (A \dx F(u))
\d x\d t
\\
&=
\int_{Q_T} \frac{1}{2}  \dt 
\left( \dx F(u) \cdot A \dx F(u)
\right)
\d x\d t
\\ & \qquad -
\int_0^T 
(
 \dt F(u) \cdot A \dx F(u)
)\bigg|_{\de I}
\d t.
\end{align*}
We evaluate the boundary term above, indeed, using equation \eqref{eq:survivingbc}, we have
\begin{align*}
\de_tF(u)\cdot A \de_x F(u) \biggr |_{\de I} 
&= \sum_{k,l= 1}^N \left( \de_t F^k(u) A^{kl}\de_xF^l(u) \right) \biggr |_{\de I}
\\
& = \de_t F^1(u(t,\ellp)) \sum_{l \neq 2} A^{1l}(\ellp)\de_xF^l(u(t,\ellp))
\\
& = a'(t) \left[ \int_I \de_t u^1 \d x - M^{12}F^2(u(t,\ellp)) \right].
\end{align*}

Thus, the left-hand side of \eqref{eq:testeq} satisfies the following inequality:
\begin{align}\label{eq:estlhs}
\ll &\geq 
\int_{Q_T}
\left[
\lambda(1-\sigma)  |\dt u|^2
+
\left(1-\frac{1}{4\sigma \lambda}\right)
|\Xi|^2
+
\frac{1}{2} \dt 
\left( A \dx F(u) \cdot \dx F(u)
\right)
\right]
\d x\d t
\nonumber \\
&\quad -
\int_0^T a'(t) 
\left[ \int_I \de_t u^1 \d x  - M^{12}F^2(u(t,\ellp)) \right]
 \d t.
\end{align}
Concerning the right-hand side of \eqref{eq:testeq}, using Young's inequality we obtain
\begin{align}\label{eq:estrhs}
 \rr &\leq 
\frac{1}{2}\left( \frac{1}{\lambda(1-\sigma)} + 
\left(1-\frac{1}{4\sigma \lambda}\right)^{-1} \right)
\int_{Q_T}
|\dx (M F(u))|^2
\d x\d t
\nonumber \\
&\quad +
\frac{1}{2}
\int_{Q_T}
\left[
 \lambda(1-\sigma)  |\dt u|^2
+
\left(1-\frac{1}{4\sigma \lambda}\right)
|\Xi|^2
\right]
\d x\d t.
\end{align}
Combining the estimates for $\ll$ and $\rr$ (i.e. \eqref{eq:testeq}, \eqref{eq:estlhs} and \eqref{eq:estrhs}), we have 
\begin{align*}
&
\int_{Q_T}
\left[
\lambda(1-\sigma)  |\dt u|^2
+
\left(1-\frac{1}{4\sigma \lambda}\right)
|\Xi|^2
\right]
\d x\d t
+
\int_I
 A \dx F(u) \cdot \dx F(u)\big|_{t=0}^{t=T}
\d x
\\
&\leq
K_{\lambda \sigma}
\int_{Q_T}
|\dx (M F(u))|^2
\d x\d t
+
\int_0^T a'(t) 
\left[ \int_I \de_t u^1 \d x  - M^{12}F^2(u(t,\ellp)) \right]
 \d t,
\end{align*}
where 
$K_{\lambda \sigma} =  \frac{1}{\lambda(1-\sigma)} + \frac{4\sigma \lambda}{4\sigma \lambda-1} $.
We recall that, for $k\geq 2$, we have $F^k(u(t,\ellm)=0$, hence, by Poincar\'e's inequality (with constant $C_P$), we obtain
$$
\int_I |F^k(u)|^2 \d x
\leq C_P 
\int_I |\de_x F^k(u)|^2 \d x, \quad \forall k\geq 2.
$$
Using Morrey's inequality (with constant $C_S$) and Proposition \ref{prop:est1}, we get
$$
\int_0^T \nm{F^2(u(t,\cdot)}_{L^\infty(I)}^2   \d t
\leq
C_S \int_0^T \nm{F^2(u(t,\cdot)}_{H^1(I)}^2   \d t
\leq (1 + C_P)C_SC_0.
$$
Consequently, we deduce that
\begin{align*}
- & \int_0^T a'(t) 
 M^{12} F^2(u(t,\ellp) )  \d t
\\
&\leq 
\nm{M^{12}}_{L^\infty(0,\infty)}
\left(
\int_0^T a'(t)^2 \d t  
\right)^{\frac{1}{2}}
 \left(\int_0^T \nm{F^2(u(t,\cdot)}_{L^\infty(I)}^2   \d t
\right)^{\frac{1}{2}}
\\
&\leq
\nm{M^{12}}_{L^\infty(0,\infty)}
a_1
\sqrt{\frac{(1 + C_P)C_SC_0}{\mu}},
\end{align*}
where $a_1$ was introduced in {\bf{[H2]}}. Thanks to Proposition \ref{prop:est1}, we also have
$$
\quad\int_{Q_T}
|\dx (M F(u))|^2
\d x\d t
\leq \nm{M}_{L^\infty(0,\infty)}^2 \frac{C_0}{\mu}.
$$
Similarly, we also have the bound
\begin{align*}
\int_{Q_T} a'(t) 
\de_t u^1 \d x \d t
&\leq 
\frac{1}{4\lambda\sigma}a_1^2
+
\lambda\sigma \int_{Q_T} 
|\de_t u^1|^2 \d x \d t.
\end{align*}
Finally, we deduce that
\begin{align*}
\quad & \lambda(1-2\sigma)  
\int_{Q_T} |\dt u|^2 \d x\d t
\\
&\qquad+
\left(1-\frac{1}{4\sigma \lambda}\right)
\int_{Q_T} |\Xi|^2 \d x\d t
+
\mu \int_I
|\dx F(u(T))|^2
\d x
\\
&\leq
K_{\lambda \sigma}
\frac{1}{4\lambda\sigma}a_1^2
+
\nm{M^{12}}_{L^\infty(0,\infty)}
a_1
\sqrt{\frac{(1 + C_P)C_SC_0}{\mu}}
\\
& \qquad
+
\nm{M}_{L^\infty(0,\infty)}^2 \frac{C_0}{\mu}
+
\frac{1}{\mu} \int_I
|\dx F(u_0)|^2
\d x,
\end{align*}
where all constants are independent of time.
Notice that, since we have obtained a bound for $\dt u$, we can use equation \eqref{eq:sysagain} and Proposition \ref{prop:est1} to deduce that $F(u)\in L^2(0,T;H^2(I))$.
Furthermore we also have a uniform estimate for $F(u)$ in $L^\infty(0,T;H^1(I))$, which implies 
$F(u) \in L^\infty((0,T)\times I)$. Since $F$ is monotone and it satisfies \eqref{eq:monotone}, this gives $u\in L^\infty((0,T)\times I)$.
In conclusion, knowing that $u$ is bounded, we deduce that the estimates for $F(u)$ lead to analogous bounds for $u$ in $L^2(0,T;H^2(I)) \cap H^1(0,T;L^2(I))$.
\end{proof}

The following technical result will be used in the proof of Theorem \ref{thm:main}.
\begin{lem}\label{lem:fractional}
Let $f: Q_T \to\RR$ be a function in $X = L^2(0,T;H^2(I)) \cap H^1(0,T;L^2(I))$, then
$$
f\in H^r(0,T;H^s(I)), \; \forall \; r,s\geq 0 
\; \text{ such that } \; r + \frac{s}{2} \leq 1,
$$
and, in turn,
$$
f\in C^{\alpha, \beta}(\bar{Q}_T)= C^{0,\alpha}([0,T];C^{0,\beta}(\bar{I})), \; \forall \; \alpha,\beta\geq 0
\; \text{ such that } \; 2\alpha + \beta \leq \frac{1}{2}.
$$
\end{lem}
\begin{proof}
Thanks to the higher order extensions for Sobolev functions, we can define $f$ on a larger rectangular domain $R\subseteq \RR^2$ containing $Q_T$. Introducing a cut-off function, we further extend $f$ to the whole space ensuring sufficiently fast decay at infinity.
Let us call $g$ such an extension and
observe that the norm of $g$ in $X'=L^2(\RR;H^2(\RR))\cap H^1(\RR;L^2(\RR))$ is controlled by the corresponding norms of $f$ on $Q_T$.
In particular, for a suitable choice of $g$, we have the inequality:
$$
\nm{g}_{X'} \leq 2\nm{f}_{X}.
$$
Let $\langle \kappa \rangle = (1+|\kappa|^2)^{1/2}$. Denoting by $(\omega, \kappa)$ the conjugate variables of $(t,x)$ in Fourier space, we have that
$$
\langle \omega \rangle \hat{g} \in L^2(\RR^2), 
\quad
\langle \kappa \rangle^2 \hat{g} \in L^2(\RR^2).
$$
This implies that 
$
\left(
\langle \omega \rangle + \langle \kappa \rangle^2
\right) \hat{g} \in L^2(\RR^2)
$
and we obtain
$$
\langle \omega \rangle^r \langle \kappa \rangle^s |\hat{g}|
\leq
\left(
\langle \omega \rangle + \langle \kappa \rangle^2
\right)^{r+\frac{s}{2}} |\hat{g}|.
$$
We obtain the desired fractional Sobolev regularity provided that $r+\frac{s}{2}\leq 1$.
The H\"older regularity follows from the standard embeddings for fractional Sobolev spaces (see e.g. \cite{dinezza2012hitchhiker}). In particular, for $r,s>\frac{1}{2}$, we take $\alpha = r - \frac{1}{2}$ and $\beta = s - \frac{1}{2}$.
\end{proof}

\begin{proof}[Proof of Theorem \ref{thm:main}]
Thanks to Theorem \ref{thm:AT}, we know that classical solutions exist for short times and that, as explained in \cite{acquistapace1988quasilinear}, they can be extended by standard methods to a maximal interval of existence denoted by $J(u_0)$.
In order to show that such solutions exist for arbitrarily large time we are going to use the criterion provided by Theorem \ref{thm:amann}.
In particular, we need H\"older continuity of $u$ with respect to time, as well as a uniform $L^\infty$ bound in the space variable.
Thanks to Proposition \ref{prop:est2} we know that $u\in L^2(0,T;H^2(I)) \cap H^1(0,T;L^2(I))$ uniformly in time. This implies that we can apply Lemma \ref{lem:fractional} and obtain uniform H\"older estimates. Thus Theorem \ref{thm:amann} allows us to conclude the proof.
\end{proof}

\section{Degenerate systems}\label{sec:degenerate}
Consider a system of equations of the type:
\begin{equation}\label{sys0}
\dt u = \dx (M(t)F(u) + A(u) \dx F(u)).
\end{equation}
Notice that this is similar to system \eqref{eq:sysgen0}, but the matrix $A(x)$ has been replaced by $A(u)$, which now plays the role of a \quot{mobility matrix} (whereas $M$ satisfies the same assumptions introduced earlier).
Such system is possibly degenerate in the sense that the matrices $A(u)$ and $DF(u)$ may vanish if $u=0$ (for all the details see Definition \ref{entropystructure} and condition \eqref{eq:coercivity} below).
Also in this case, we will prove global existence of solutions under suitable assumptions on $F$. We will show that entropy methods developed in recent years, and, in particular, the so-called \quot{boundedness-by-entropy principle} presented in \cite{jungel2015boundedness}, can be applied without major modifications.

\begin{defi}[Entropy structure]\label{entropystructure}
We say system \eqref{sys0} has an \emph{entropy structure} if
there exists a function $H:\RR^N\to\RR$ such that
\begin{itemize}
\item $H$ is a convex function of class $C^{2}$ and it defines the following
\emph{entropy functional} 
$$
E[u]=\int_{\Omega}H(u) \d x.
$$
\item the map $F(\cdot) = DH(\cdot)$ (i.e. the gradient of $H$) defines a change of coordinates (bi-Lipschitz diffeomorphism) from an open and connected domain $U\subset\RR^{N}$ into the whole $\mathbb{R}^{N}$.
\end{itemize}
\end{defi}
\begin{defi}[Weak formulation]
\label{def:weaksolsys}
We say that the vector function $u$ is a weak solution of \eqref{sys0} subject to the Dirichlet boundary condition $u(t,\ellpm)=0$, for a.e. $t>0$, if
$$
u\in L^{2}(Q_T),
\quad
A(u)\de_x F(u)\in L^{2}(Q_T)
\quad
\partial_{t}u\in L^{2}(0,T;(H^{1}(I))').
$$
and, for any test function $\eta\in C_0^\infty(I)$ and a.e. $t\geq 0$, it holds
$$
\langle \de_t u, \eta \rangle 
+
\int_{\Omega}
\left[
A(u)\dx F(u) \cdot \dx \eta
-
\dx M F(u) \cdot \eta
\right]
\d x \d t
=0,
$$
where $\langle\cdot,\cdot\rangle$ indicates the duality pairing.
Moreover we require $u(t,\cdot)\to u_0(\cdot)$ in $H^{1}(I)'$ as $t\to 0$.
\end{defi}

\begin{rem}[Entropy decay]
The new unknown $w\in\RR^N$ obtained setting $w=DH(u)=F(u)$, for $u\in U$, is commonly referred to as the \emph{entropy variable}. The domain $U$ (from Definition \ref{entropystructure}) is typically a bounded Lipschitz subset of $\RR^N$.
We consider the boundary condition $F(u)\big|_{\de I} = 0$, which implies $w\big|_{\de I} = 0$.
Using Definition \ref{entropystructure}, we will see that for solutions of \eqref{sys0} in the sense of Definition \ref{def:weaksolsys} we have $\frac{dE}{dt}\leq 0$.
\end{rem}
We now present the main existence result of this section.
\begin{theor}[Boundedness-by-entropy principle, \cite{jungel2016entropy}]
\label{thm:existsys-with-entropy} 
Consider problem \eqref{sys0} with boundary condition $u(t,\ellpm)=0$ for a.e. $t>0$, let $U$ be an open and bounded subset of ${\mathbb{R}}^{n}$ and suppose $u_0\in U$. Consider the following hypotheses:
\begin{enumerate}
\item 
There exist $\gamma_1,\gamma_2\in\RR$ such that $\gamma_1<\gamma_2$ and $U \subset(\gamma_1,\gamma_2)^{N}$.
Furthermore, there exist $\alpha_{i}^{*},m^{i}\geq0$ ($i=1...N$) such
that for any vector $\xi\in\RR^N$ and any $u\in U$ 
\begin{equation}\label{eq:coercivity}
[ A(u)\he(H)(u) ]
\xi\cdot \xi
\geq\sum_{i=1}^{N}\alpha_{i}(u^{i})^{2}(\xi^{i})^{2},
\end{equation}
where $\alpha_{i}(u^{i})$ coincides either with $\alpha_{i}^{*}(u^{i}-\gamma_1)^{m^{i}-1}$
or with $\alpha_{i}^{*}(\gamma_2-u_{i})^{m_{i}-1}$.
\item 
We have $A\in C^0(\bar{U};\RR^{N\times N})$ and there exists $L>0$ such that, for all $u\in U$
and all $i.j=1...N$ for which $m_{j}>1$, it holds
$
|A(u)\he(H)^{ij}(u)|\leq L |\alpha_{j}(u^{j})|.
$

\item 
It holds $u_{0}(x)\in U$ for a.e. $x\in I$. 
\end{enumerate}
Then there exists a bounded weak solution $u\in\bar{U}$ of problem \eqref{sys0} in the sense of Definition \ref{def:weaksolsys} for all $t>0$.
\end{theor}
\begin{proof}
The proof is analogous to the one given in \cite{jungel2015boundedness}, the only differences consist in the presence of Dirichlet boundary conditions (instead of no-flux) and in the first order terms (which were not present in the original proof).
Neither of these variations affects the argument in a significant way.
In particular, the first order terms do not contribute to the estimates since, once we change variables to $w=F(u)$, and we test against $w$ in the weak formulation, such terms vanish. In particular, we have:
\begin{align*}
\frac{dE}{dt}=
\int_{Q_T} \dt u\cdot w \d x\d t
&= \int_{Q_T} \left[
\dx (M w)\cdot w + \dx  (A(u) \dx w)\cdot w \right]\d x\d t
\\
&= \int_{Q_T} \dx 
\left[
\left(\frac{1}{2} M w\cdot w\right) -   A(u) \dx w\cdot \dx w \right] \d x\d t
\\
&= - \int_{Q_T}   A(u) \dx w\cdot \dx w \d x\d t,
\end{align*}
which is the key estimate in \cite{jungel2015boundedness}.
The rest of the proof follows without major modifications.
\end{proof}

\bibliographystyle{plain}

\end{document}